\documentclass{amsart}
\usepackage{amssymb,latexsym}
\usepackage{amsmath}
\usepackage{amscd}
\usepackage{graphicx}
 \usepackage{color}
\usepackage{enumerate}
\numberwithin{equation}{section}
\theoremstyle{plain}
 \newtheorem{theorem}{Theorem}[section]
 \newtheorem{lemma}[theorem]{Lemma}

 \newtheorem{corollary}[theorem]{Corollary}
\theoremstyle{definition}
 \newtheorem{definition}[theorem]{Definition}
 \newtheorem{remark}[theorem]{Remark}

\newenvironment{enumeratei}{\begin{enumerate}[\quad\upshape (i)]} {\end{enumerate}}

\newcommand \aijk [3] {\alpha(#1,#2,#3)}
\newcommand \pnul {q}
\newcommand \pegy {p}
\newcommand \fa {f_\alpha}
\newcommand \fb {f_\beta}
\newcommand \hha { h_{\alpha}}
\newcommand \hhb { h_{\beta}}
\newcommand \lexless[1] {<_{\sscr{\textup{lex}}}}
\renewcommand \phi{\varphi}
\newcommand \alg [1] {\mathfrak #1}
\newcommand \restrict [2] {{#1}\kern-1pt \rceil_{\kern-1pt #2}}
\newcommand \rhullo [1] {\textup{Conv}_{\real^{#1}}} 
\newcommand \Pow {\textup{Pow}}
\newcommand \Pnt {\textup{Points}}
\newcommand \Tany {T}
\newcommand \Trr {T_{\textup{RR}}}

\newcommand \Tczk {T_{\textup{new}}}
\newcommand \Tellips {T_{\textup{ellipses}}}
\newcommand \Tcircle {T_{\textup{circles}}}
\newcommand \Hull [1]{\textup{Conv}_{#1}}
\newcommand \Nnull {\mathbb N_0}
\newcommand \Nat {\mathbb N}
\newcommand \defiff {\overset{\textup{def}}\iff}
\newcommand \acirc[2] {C(#1,#2)}
\newcommand\set [1]{\{#1\}}
\newcommand \tuple [1] {\langle #1 \rangle}
\newcommand \pair [2] {\tuple{#1,#2}}
\newcommand \Bpair [2] {\Big\langle \kern-2pt #1,#2 \Big\rangle }

\newcommand \real {\mathbb R}

\DeclareMathOperator \Pfin {\textup{Pow}_{\textup{fin}}}
\newcommand \sscr [1] {\scriptscriptstyle{#1}}
\newcommand \flft {{^{\sscr{\textup {lf}}}\kern-2pt g}}
\newcommand \frht {{^{\sscr{\textup {rg}}}\kern-1.5pt g}}
\newcommand \Gaux {{^{\sscr{\textup {aux}}}\kern -1.2 pt G }}
\newcommand \Ftemp  {{^{\sscr{\textup {set}}}\kern -2.2 pt F }}
\newcommand \Fczk  {{^{\sscr{\textup {gf}}}\kern -2.2 pt F }}
\newcommand \graph [1] {\textup{Graph}(#1)}
\newcommand \dom [1] {\textup{Dom}(#1)}
\newcommand \id {\textup{id}}
\newcommand \vvec [1] {\vec #1\kern1pt }
\newcommand \nothing [1] {}
\newcommand \red [1] {{\color{red}#1\color{black}}}
\newcommand \mpar [1] {}

\newcommand \tbf[1] {\textbf{#1}}

\begin{document}
\title
{Representing convex geometries by almost-circles}

\author[G.\ Cz\'edli]{G\'abor Cz\'edli}
\email{czedli@math.u-szeged.hu}
\urladdr{http://www.math.u-szeged.hu/\textasciitilde{}czedli/}
\address{University of Szeged\\ Bolyai Institute\\Szeged,
Aradi v\'ertan\'uk tere 1\\ Hungary 6720}
\author[J.\ Kincses]{J\'anos Kincses}
\email{kincses@math.u-szeged.hu}
\urladdr{http://www.math.u-szeged.hu/\textasciitilde{}kincses/}
\address{University of Szeged\\ Bolyai Institute\\Szeged,
Aradi v\'ertan\'uk tere 1\\ Hungary 6720}

\thanks{This research was supported by
NFSR of Hungary (OTKA), grant number K 115518}

\begin{abstract} 
\emph{Finite convex geometries} are combinatorial structures. 
It follows from a recent result of M.\ Richter and L.G.\ Rogers that there is an infinite set $\Trr$ of planar convex polygons such that $\Trr$ with respect to geometric convex hulls is a locally convex geometry and  every finite convex geometry can be represented by restricting the
structure  of $\Trr$ to a finite subset in a natural way. 
An \emph{almost-circle of accuracy} $1-\epsilon$ is a differentiable convex simple closed curve $S$ in the plane having an inscribed circle of radius $r_1>0$ and a circumscribed circle of radius $r_2$ such that the ratio $r_1/r_2$ is at least   $1-\epsilon$. 
Motivated by Richter and Rogers'  result, we construct a set $\Tczk$ 
such that (1) $\Tczk$ contains  all points of the plane as degenerate singleton circles and all of its non-singleton members are  differentiable convex simple closed planar curves;  (2)  $\Tczk$ with respect to the geometric convex hull operator is a locally convex geometry; (3) as opposed to $\Trr$,  $\Tczk$ is closed with respect to non-degenerate affine transformations; and 
(4) for every (small) positive $\epsilon\in\real $ and for every finite convex geometry, there are continuum many pairwise affine-disjoint finite subsets $E$ of $\Tczk$ such that each $E$ consists of almost-circles of accuracy $1-\epsilon$ and the convex geometry in question is  represented by restricting the convex hull operator to $E$. The affine-disjointness of $E_1$ and $E_2$ means that, in addition to $E_1\cap E_2=\emptyset$, even $\psi(E_1)$ is disjoint from $E_2$ for every non-degenerate affine transformation $\psi$.
\end{abstract}



\dedicatory{Dedicated to the eighty-fifth birthday of B\'ela Cs\'ak\'any}


\date{\red{August 23, 2016}}

\maketitle

\section{Introduction}

For a set $E$, let $\Pow(E)=\set{X:X\subseteq E}$ and $\Pfin(E)=\{X: X\subseteq E$ and $X$ is finite$\}$ denote the \emph{powerset} and the \emph{set of finite subsets} of $E$, respectively.
Convex geometries are defined as follows. 

\begin{definition}[Adaricheva and Nation~\cite{kirajbbooksection} and \cite{kajbn}]\label{defconvgeo}
A pair $\tuple{E;\Phi}$ is a \emph{convex geometry}, also called an \emph{anti-exchange system}, if it satisfies the following properties:
\begin{enumeratei}
\item\label{defconvgeoa} $E$ is a set, called the set of \emph{points}, and $\Phi\colon \Pow(E)\to \Pow(E)$ is a \emph{closure operator}, that is, for all $X\subseteq Y\subseteq E$, we have $X\subseteq\Phi(X)\subseteq \Phi(Y)=\Phi(\Phi(Y))$.
\item\label{defconvgeob} If $A\in \Pow(E)$, $x,y\in E\setminus \Phi(A)$, and   $ \Phi(A\cup\set x) = \Phi(A\cup\set y)$, then $x=y$. (This is the so-called \emph{anti-exchange property}.)
\item\label{defconvgeoc}  $\Phi(\emptyset)=\emptyset$.
\end{enumeratei}
\end{definition}

Although local convexity is a known concept for topological vector spaces,  the following concept seems to be new.

\begin{definition}\label{defloconvgeo} A pair $\tuple{E;\Phi}$ is a \emph{locally convex geometry} if \ref{defconvgeo}\eqref{defconvgeoa}, \ref{defconvgeo}\eqref{defconvgeoc}, and
\begin{enumeratei}
\setcounter{enumi}{3}
\item\label{defloconvgeod} If $X\in \Pfin(E)$, $d,d'\in E\setminus \Phi(X)$, and   $ \Phi(X\cup\set d) = \Phi(X\cup\set {d'})$, then $d=d'$. 
\end{enumeratei}
(This condition will be called the \emph{local anti-exchange property}.)
\end{definition}

For example, if   $\rhullo n$ denotes the usual convex hull operator in the space $\real ^n$, 
\begin{equation}
\text{then $\tuple{\real ^n;\rhullo n}$ is a convex geometry.}
\label{eqtxtclScGoM}
\end{equation}
Every convex geometry is a locally convex geometry but \eqref{eqtxtlocnconvlPs}  we will soon show that the converse implication fails. 
Given a convex or a locally convex geometry $\tuple{E;\Phi}$ and a subset $E_0$ of $E$, we can consider the \emph{restriction}
\begin{equation}
\parbox{9cm}{$\restrict{\tuple{E;\Phi}}{E_0}:=\tuple{E_0;\Phi_0}$, where the map
$\Phi_0\colon \Pow(E_0)\to \Pow(E_0)$ is defined by the rule $\Phi_0(X):=E_0\cap \Phi(X)$,}
\end{equation}
of the original convex geometry to its subset, $E_0$. This terminology is justified by the following statement; without the trivial modification of adding ``locally'', it is taken from {Edelman and Jamison~\cite[Thm.\ 5.9]{edeljam}}.

\begin{lemma}[{Edelman and Jamison~\cite[Thm.\ 5.9]{edeljam}}]\label{lemmarestrgeom}
If   $\tuple{E;\Phi}$ is a convex or locally convex geometry, then so is its restriction $\restrict{\tuple{E;\Phi}}{E_0}$, for every subset $E_0$ of $E$.
\end{lemma} 

Since our setting is slightly different and the proof is very short, we will prove this lemma in Section~\ref{sectprtools} for the reader's convenience. Note that a finite locally convex geometry is automatically a convex geometry. As an additional justification of our terminology, we mention the following statement even if its proof, postponed to Section~\ref{sectprtools}, is trivial.

\begin{lemma}\label{lemmaTfCLoCc}
A pair  $\tuple{E;\Phi}$ of a set $E$ and a closure operator $\Phi\colon \Pow(E)\to \Pow(E)$ on $E$ is a locally convex geometry if and only if its restriction $\restrict{\tuple{E;\Phi}}{E_0}$ is a convex geometry for every \emph{finite} subset $E_0$ of $E$.
\end{lemma} 

%
%
%

\emph{Finite} convex geometries are intensively studied mathematical objects.  There are several  combinatorial and lattice theoretical ways to characterize and describe these objects; see, for example, Adaricheva and Cz\'edli~\cite{kaczg}, Avann~\cite{avann1},  Cz\'edli~\cite{czgcoord}, Dilworth~\cite{dilworth40}, Duquenne~\cite{duquenne}, Edelman and Jamison~\cite{edeljam}, and see also  Adaricheva and Nation~\cite{kirajbbooksection} and Monjardet~\cite{monjardet} for surveys.
Natural and easy-to-visualize examples for finite convex geometries are 
obtained by considering
the restrictions $\restrict{\tuple{\real ^n;\rhullo n}}{E}$ of $\tuple{\real ^n;\rhullo n}$ to finite sets $E\subseteq \real ^n$ of points, for $n\in\set{1,2,3,\dots}$. Note that most of the finite convex geometries are not isomorphic to any of these restrictions.
The first result that represents \emph{every} finite convex geometry with the help of $\tuple{\real ^n;\rhullo n}$ was proved in 
Kashiwabara,  Nakamura, and Okamoto~\cite{kashiwabaraatalshelling}. This result uses auxiliary points and has not much to do with 
restrictions in our sense, so we do not give further details on it.

Next, let $T$ be a set of  subsets of the plane. For $X\subseteq \Tany$, we can naturally define 
\begin{equation}
\begin{aligned}
\Pnt(X)&=\bigcup_{C\in X}C,\quad \text{ and} \cr
\Hull\Tany(X):&=\set{D\in \Tany: D\subseteq \rhullo2(\Pnt(X))}.
\end{aligned}
\label{eqconvTanydef}
\end{equation} 
The notation $\Pnt$ is self-explanatory; $\Pnt(X)$ is the set of points of the members of $X$.
Note the difference between the notations $\rhullo2$ and $\Hull\Tany$; the former applies to sets of \emph{points} and yields a set of points while the latter to sets of sets and yields a set of sets. Typically in the present paper, $\Tany$ consists of closed curves and  we apply  $\Hull\Tany$ to \emph{sets of closed curves}. 

In lucky cases but \emph{far from always}, the structure $\tuple{\Tany;\Hull\Tany}$ is a locally convex geometry. For example, if $\Tany =\real ^2$, then  $\tuple{\Tany;\Hull\Tany}= \tuple{\real ^2;\rhullo 2} $ is even a convex geometry.  In order to obtain a more interesting example, let $\Tellips$ be the set of all non-flat ellipses in the plane. Here, by a \emph{non-flat} ellipse we mean an ellipse that is either of positive area or it consists of a single point. We have the following observation.
\begin{equation}
\parbox{7cm}{$\tuple{\Tellips;\Hull\Tellips}$ is a locally convex geometry. However, it is not a convex geometry.}
\label{eqtxtlocnconvlPs}
\end{equation}
The first part of \eqref{eqtxtlocnconvlPs} follows from  Cz\'edli~\cite{czgcircles}, where the argument is formulated only for circles but it clearly holds for ellipses. This part will also follow easily from the present paper; see the proof of part \eqref{thmmainb} of Theorem~\ref{thmmain}. In order to see the second part, let $\mathbb Z$ stand for the set of integer numbers, and let
 $C_k=\set{\pair x y: (x-k)^2+y^2=1}$, $X=\set{\pair x y: x^2 + (y-3)^2=1}$, and $Y=\set{\pair x y: x^2 + (y-2)^2=4}$. Then the anti-exchange property, \ref{defconvgeo}\eqref{defconvgeob}, fails for $A=\set{C_k: k\in \mathbb Z}$, $X$, and $Y$. 

Related to \cite[(4.6)]{czgcircles}, it is an open problem 
\begin{equation}
\parbox{8cm}{whether every finite convex geometry can be represented in the form  $\restrict{\tuple{\Tellips;\Hull\Tellips}}E$;}
\label{eqtxtproblellips}
\end{equation}
up to isomorphism, of course. The answer is affirmative for finite convex geometries of convex dimension at most 2, to be defined later; the reason is that, denoting the set of all circles (including the singletons)  of the plane by   $\Tcircle$, \cite{czgcircles} proves that 
\begin{equation}
\parbox{8.5cm}{every finite convex geometry of convex dimension at most $2$ is isomorphic to some $\restrict{\tuple{\Tcircle;\Hull\Tcircle}}{E}$.}
\label{eqtxtcRls}
\end{equation}
Actually,  \cite{czgcircles} proves a bit more but the details are irrelevant here. 

We obtain from Richter and Rogers~\cite[Lemma 3]{richterrogers}, or in a straightforward way, that for every set $\Tany$ of pairwise vertex-disjoint planar convex polygons, $\tuple{\Tany;\Hull\Tany}$ is a locally convex geometry. Therefore, since there are only countably many isomorphism classes of finite convex geometries, \cite[Theorem]{richterrogers} implies that 
\begin{equation}
\parbox{9cm}{there exists a set $\Trr$ of pairwise vertex-disjoint convex polygons in the plane such that $\tuple{\Trr;\Hull\Trr}$ is a locally convex geometry and every finite convex geometry
can be represented as some of its restrictions,   $\restrict{\tuple{\Trr;\Hull\Trr}}{E}$.}
\label{eqtxtRRskTrr}
\end{equation}

\begin{remark}\label{remdSk} In this paper, the ``elements'' of our convex geometries are closed lines, mostly \emph{simple closed curves}; for example, they are  circles in \eqref{eqtxtcRls}. This setting is more natural here, since we will work with curves. However, it would be an equivalent setting to replace these ``elements'' by their convex hulls. For example, we could consider \emph{closed disks} in \eqref{eqtxtcRls} instead of circles, and similarly in \eqref{eqtxtproblellips}, \eqref{eqtxtRRskTrr}, and the forthcoming Theorem~\ref{thmmain}.  However, instead of doing so, we require only that our simple closed curves should be \emph{convex}, that is, they should coincide with the boundaries of their convex hulls.
\end{remark}

Every non-degenerate affine transformation $\psi$ of the plane, that is, every map $\psi\colon\real^2\to\real^2$ defined by $\pair xy \mapsto \pair xy A + \pair {b_1}{b_2}$, where $A$ is a 2-by-2 matrix with nonzero determinant, 
is known to induce an automorphism of the convex geometry $\tuple{\real ^2;\rhullo 2}$. Furthermore, if convex linear combinations are considered operations, then there are no more automorphisms, say, by Cz\'edli, Mar\'oti, and Romanowska~\cite[Theorem 2.4]{czmmaroman}.
Hence and because of Remark~\ref{remdSk}, it is a natural desire to replace $\Trr$ in \eqref{eqtxtRRskTrr} by a set of convex simple closed planar curves that is closed with respect to non-degenerate affine transformations. Note that $\Trr$ is not even closed with respect to parallel shifts. Actually, except from trivial cases, if $\Tany$ is a set of polygons closed with respect to parallel shifts, then $\tuple{\Tany,\Hull\Tany}$ is not a locally convex geometry in general; see Figure~\ref{figtrrnot} for an explanation.

\begin{figure}[ht] 
\centerline
{\includegraphics[scale=1.0]{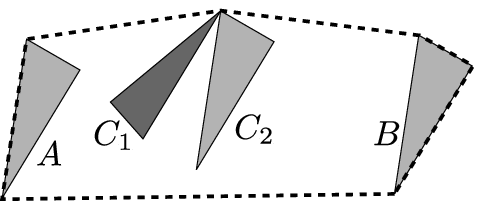}}
\caption{$\Hull\Tany(\set{A,B}\cup\set{C_1})=\Hull\Tany(\set{A,B}\cup\set{C_2})$,  $\set{C_1,C_2}\cap\Hull\Tany(\set{A,B})=\emptyset$, but $C_1\neq C_2$.}
\label{figtrrnot}
\end{figure}

Our goal is to replace $\Trr$ in 
\eqref{eqtxtRRskTrr} with a  set of differentiable convex simple closed planar curves that is closed with respect to non-degenerate affine transformations. Although $\Tellips$ is closed, \eqref{eqtxtproblellips} remains an open problem. On the other hand, Figure~\ref{figtrrnot} shows that we cannot replace $\Trr$ with a set of polygons in \eqref{eqtxtRRskTrr}. Therefore, we are going to modify the circles in \eqref{eqtxtcRls} slightly so that the restriction on the convex dimension could be removed. Of course, the non-degenerate affine transformations will bring ellipse-like closed curves in  besides the circle-like ones.  Our definition of a circle-like closed curve is the following one.

\begin{definition}\label{defAmstCrlE}
For a nonnegative real number  $\epsilon<1$ and a differentiable convex simple closed planar curve $S$, we say that $S$ is an \emph{almost-circle of accuracy} $1-\epsilon$ if  $S$ has an inscribed circle of radius $r_1>0$ and a circumscribed circle of radius $r_2$ such that the ratio $r_1/r_2$ is at least   $1-\epsilon$. 
\end{definition}

Our convention in the paper is that $0\leq \epsilon <1$, even if this will not be repeatedly mentioned. Following the tradition, we think that $\epsilon$ is very close to 0.  The case $\epsilon=0$ occurs only for non-degenerate circles, which are of  accuracy 1.
Note the following feature of our terminology: if $1-\epsilon' < 1-\epsilon$, that is,   $\epsilon'>\epsilon$,  then every  almost-circle of accuracy $1-\epsilon$ is also an almost-circle of accuracy $1-\epsilon'$.

\begin{definition} For $E_1,E_2\subseteq \Pow(\real)$, $E_1$ and $E_2$ are \emph{affine-disjoint} if for every $X_1\in E_1$ and every non-degenerate affine transformation $\psi\colon \real^2\to\real^2$, $\psi(X_1)\notin E_2$.
In other words,  $E_1$ and $E_2$ are \emph{affine-disjoint} if $\psi(E_1)\cap E_2=\emptyset$ for all  $\psi$ as above.
\end{definition}

Note that affine disjointness is a symmetric relation, since the inverse of $\psi$ above is also a non-degenerate affine transformation. 
Now, we are in the position to formulate the main result of the paper.

\begin{theorem}[Main Theorem]\label{thmmain}
There exists a set $\Tczk$ of some subsets of the plane, that is, $\Tczk\subseteq \Pow(\real^2)$, with the following properties.
\begin{enumeratei}
\item\label{thmmaina} Every non-singleton member of $\,\Tczk$ is a differentiable convex simple closed planar curve, and  for all $\vec p\in\real^2$, the singleton $\set{\vec p}$ belongs to $\Tczk$.
\item\label{thmmainb} $\tuple{\Tczk;\Hull\Tczk}$ is a locally convex geometry.
\item\label{thmmainc} $\Tczk$ is closed with respect to non-degenerate affine transformations.
\item\label{thmmaind} For every finite convex geometry $\tuple{E_0;\Phi_0}$ and for every (small) positive real number $\epsilon<1$, there exist continuum many pairwise affine-disjoint finite subsets $E$ of $\Tczk$ such that $\tuple{E_0;\Phi_0}$ is isomorphic to the restriction $\restrict{\tuple{\Tczk;\Hull\Tczk}}E = \tuple{E;\Hull E}$ and $E$ consists of non-degenerate almost-circles of accuracy $1-\epsilon$.
\end{enumeratei}
\end{theorem}

\begin{remark} Clearly, the restriction $\restrict{\tuple{\Tczk;\Hull\Tczk}}{\set{\set{\vec p}: \vec p\in\real^2}}$ is isomorphic to the classic $\tuple{\real^2;\rhullo2}$, see \eqref{eqtxtclScGoM}, since the map defined by $\set{\vec p}\mapsto \vec p$ is an isomorphism. Thus, we can view $\tuple{\Tczk;\Hull\Tczk}$ as an extension of the plane with many large ``unconventional points''.
\end{remark}

\begin{remark}\label{remconTnm} Necessarily, the cardinality of  $\Tczk$ in  Theorem~\ref{thmmain} is  continuum, that is, $2^{\aleph_0}$.  Furthermore, the theorem is sharp in the sense that 
neither  $|\Tczk|$, nor ``continuum many'' in part ~\eqref{thmmaind}  could be larger.  
\end{remark}

Having no reference at hand, we follow a well-known argument below.

\begin{proof}[Proof of Remark~\ref{remconTnm}]  
Take a differentiable curve $\set{\pair{x(t)}{y(t)}:t\in [a,b]}$. There are continuum many $a$ and $b$. Since every continuous function $[a,b]\to \real$ is determined by its restriction to $[a,b]\cap\mathbb Q$, there are continuum many such functions $x\colon [a,b]\to \mathbb R$, and the same holds for the functions $y\colon [a,b]\to \mathbb R$. Hence, $|\Tczk|$ is at most continuum, whereby the statement of the remark holds.
\end{proof}

The following statement shows that if we disregard the singletons, then a somewhat weaker statement can be formulated in a slightly simpler way.  Since the statement below follows from Theorem~\ref{thmmain} trivially, there will be no separate proof for it.

\begin{corollary}\label{corolhtmshdnB}
There exists a set $\Tczk$ of \emph{differentiable}  convex simple closed  planar curves satisfying \textup{\ref{thmmain}\eqref{thmmainb}}, \textup{\ref{thmmain}\eqref{thmmainc}}, and \textup{\ref{thmmain}\eqref{thmmaind}}. 
\end{corollary}

\section{Affine-rigid functions}
As usual, for $S\subseteq \real$ and a real function  $f\colon S\to \real$, the \emph{domain} $S$ of $f$ is denoted by $\dom f$. The \emph{graph} of $f$ is denoted by
\[\graph f:=\set{\pair x{f(x)}: x\in\dom f}.
\]
A \emph{proper interval} of $\real$ is an interval of the form $[a,b]$, $(a,b]$, $[a,b)$, or $(a,b)$ such that $a<b\in\real\cup\set{-\infty,\infty}$. 
Even if this is not repeated all the times, this paper assumes that $\dom f$ of an arbitrary real function is a proper interval and that $f$ is differentiable on $\dom f$. 
(If $\dom f=[a,b)$ or $\dom f=[a,b]$, then the differentiability of $f$ at $a$ is understood from the right, and similarly for $b$.) As a consequence of our assumption, $\graph f$ is a smooth curve. For an affine transformation $\psi\colon \real ^2\to \real ^2$, the $\psi$-image of $\graph f$ is, of course, $\set{\psi(\pair x{f(x)}): x\in\dom f}$. 
By an (open) \emph{arc} of a curve we mean a part of the curve (strictly) between two given points of it. Note that in degenerate cases, an arc can be a  \emph{straight line segment}; this possibility will not occur for the members of $\Tczk$.

\begin{definition} A set $G$ of real functions is said to be \emph{affine-rigid} if 
whenever $g_1$ and $g_2$ belong to $G$, $\psi_1\colon \real ^2\to \real ^2$ and
$\psi_2\colon \real ^2\to \real ^2$ are \emph{non-degenerate} affine transformations, and the
curves $\psi_1(\graph{g_1})$ and  $\psi_2(\graph{g_2})$ have a (small) nonempty open arc in common,
then $g_1=g_2$ and $\psi_1=\psi_2$.
\end{definition}

Note that rigidity (with respect to a given family of maps) is a frequently studied concept in various fields of mathematics; we mention only  \cite{czgrigid} and \cite{czgdjsrigid} from 2016, when the present paper was submitted.
As usual, $\Nnull$ stands for the set $\set{0,1,2,\dots}$ of non-negative integers and $\Nat=\Nnull\setminus\set 0$. 

\begin{remark} Affine-rigidity is a strong assumption even on a singleton set $\set{g}$. 
For example, if $n\in\Nat$ and we define $g\colon \real\to \real$ by $g(x)=x^n$, then $\set{g}$ fails to be affine-rigid.
In order to see this, let $g_1=g_2=g$, let $\psi_1$ be the identity map, and let 
$\psi_2\colon\real\to\real$ be defined by $\pair x y\mapsto \pair {cx}{c^{n}y}$, where $c\in\real\setminus\set 0$ is a constant.
Then  $\psi_1(\graph{g_1}) = \psi_2(\graph{g_2})$ but $\psi_1\neq \psi_2$.
\end{remark}

The following lemma is easy but it will be important for us.

\begin{lemma}\label{lemmalgnDfFcrFf} A set $G$ of real functions is affine-rigid if and only if for all $g_1,g_2\in G$ and for every non-degenerate affine transformation $\psi:\real^2\to\real^2$, if $\psi(\graph{g_1})$ and $\graph{g_2}$ have a  nonempty open arc in common,
then $g_1=g_2$ and $\psi$ is the identity map.
\end{lemma}

\begin{proof} First, assume that $G$ is affine-rigid. Letting $\psi_2$ be the identity transformation and applying the definition of affine-rigidity, it follows that the condition given in the lemma holds. 

Second, assume that $G$ satisfies the condition given in the lemma. Let $g_1,g_2\in G$, and 
let $\psi_1$ and $\psi_2$ be non-degenerate affine transformations such that  $\psi_1(\graph{g_1})$ and $\psi_2(\graph{g_2})$ have a nonempty open arc  in common. Composing maps from right to left, let $\psi=\psi_2^{-1}\circ \psi_1$. That is, for every point $\vec p\in\real^2$, $\psi(\vec p)= \psi_2^{-1}(\psi_1(\vec p))$. Clearly,  $\psi(\graph{g_1})$ and $\graph{g_2}$ have a nonempty open  arc  in common. By the assumption, $g_1=g_2$ and $\psi=\id_{\real^2}$, the identity transformation on $\real^2$. The second equality gives that $\psi_1=\psi_2$, proving the affine-rigidity of $G$.
\end{proof}

Next, we define the following polynomials and consider them functions $[0,1]\to\real$:
\begin{align}
 \pegy (x)&=x(1-x)(x^5-x^4+1)=-x^7+2x^6-x^5-x^2+x,\cr
 \pnul (x)&=x(1-x)=-x^2+x,\text{ and, for }\alpha\in(0,1),\cr
\fa(x)&=\alpha \pegy (x)+(1-\alpha)\pnul (x)=-\alpha x^7+2\alpha x^6-\alpha x^5-x^2+x.
\label{eqfalphdF}
\end{align}

\begin{lemma}\label{lemmaFczkmjt} The set $\Fczk=\set{f_\alpha: \alpha\in(0,1)}$ of $[0,1]\to\real$ functions has the following properties.
\begin{enumerate}[\quad\upshape (F1)]
\item\label{Ffa} For all $\alpha\in(0,1)$,  $f_\alpha$ is twice differentiable on $(0,1)$, and it is differentiable at $0$ and $1$  from right and left, respectively. 
\item\label{Ffb} For all $\alpha\in(0,1)$ and $x\in(0,1)$, $f_\alpha''(x)<0$; note that this condition and \textup{(F\ref{Ffa})} imply that $f_\alpha $ is strictly concave on $[0,1]$.
\item\label{Ffc} For all  $\alpha\in(0,1)$ , we have $f_\alpha(0)=f_\alpha(1)=0$, $f'_\alpha(0)=1$, and $f'_\alpha(1)=-1$.
\item\label{Ffd} For all  $0<\alpha<\beta<1$ and $x\in(0,1)$, we have that $f_\alpha (x)> \fb(x)$.
\item\label{Ffe} $\Fczk$ is an affine-rigid set of functions.
\item\label{Fff} For all $\alpha\in (0,1)$ and  $x\in(0,1)$, we have  that $0<f_\alpha (x)<1/2-|x-1/2|$. 
\end{enumerate} 
\end{lemma}

In the notation $\Fczk$,  the superscript comes from ``good functions''.
 It is only the properties (F\ref{Ffa})--(F\ref{Fff}) of $\Fczk$ that we will need. Certainly, many sets of functions  parameterized with  $\alpha\in(0,1)$ have these properties; we have chosen our $\Fczk$ because of its simplicity. Note that (F\ref{Fff}), whose only role is to explain the connection of $f_\alpha $ to the triangle $\Delta(U,V,W)$ in Figure~\ref{figfunc}, is a consequence of (F\ref{Ffa})--(F\ref{Ffc}); this implication will be explained in the proof. Note also that  $|\fa(x)-\fb(x)|$ is small. Hence, in order to make our figures more informative, the graphs of the $\fa\in \Fczk$ are not depicted precisely. However, (F\ref{Ffa})--(F\ref{Ffd}) and (F\ref{Fff}) are faithfully shown by the figures.

\begin{proof}[Proof of Lemma~\ref{lemmaFczkmjt}] (F\ref{Ffa}) is trivial, since the functions $\fa$ are polynomial functions. Consider the auxiliary function $a(x):=-20x^5+20x^4$. Using that $4/5$ is the only root of $a'(x)$ in $(0,1)$, it is routine to see that $a(x)$ takes its maximum on $[0,1]$ at $x=4/5$ and this maximum is $a(4/5)<2$. (Actually, $a(4/5)=1.6384$ but we do not need the exact value.) Hence, for all $x\in(0,1)$,
$a(x)-2<0$, whereby
\begin{align*}
\pegy''(x)&=-42x^5+60x^4-20x^3-2\cr
 &= -20x^5+20x^4 -2 - 2 x^5 - (20x^5 -40 x^4 + 20x^3) \cr
&= (a(x)-2) - 2 x^5 - 20x^3(x-1)^2<0.
\end{align*}
Also, $\pnul''(x)=-2$ is negative. Thus, $\fa''(x)=\alpha\pegy''(x)+(1-\alpha)\pnul''(x) <0$ for $x\in(0,1)$, proving (F\ref{Ffb}). From $\pegy'(0)=\pnul'(0)=1$ and $\pegy'(1)=\pnul'(1)=-1$, we conclude (F\ref{Ffc}). Observe that, for $x\in(0,1)$, $\pnul(x)-\pegy(x)= x^7-2x^6-x^5=x^5(x-1)^2>0$.  Hence, for $0<\alpha<\beta<1$ and $x\in(0,1)$,
\begin{align*}
\fa(x)-\fb(x)&=\alpha\pegy(x)+(1-\alpha)\pnul(x) - \beta\pegy(x)-(1-\beta)\pnul(x)\cr
&=(\beta-\alpha)(\pnul(x)-\pegy(x))>0,
\end{align*}
which proves (F\ref{Ffd}).

Next, in order to prove (F\ref{Ffe}), it  suffices to verify the condition given in Lemma~\ref{lemmalgnDfFcrFf}. In order to do so, let $\psi\colon\real^2\to\real^2$ be a non-degenerate affine transformation,
and assume that $\fa,\fb\in\Fczk$ such that $\psi(\graph{\fa})$ and $\graph{\fb}$ have a nonempty open  arc  in common. 
 It is well-known that $\psi$ is given by the following rule
\begin{equation}
\pair \xi\eta\mapsto \pair \xi\eta A + \pair{b_1}{b_2}\text{, where }A=
\begin{pmatrix}
a_{11}&a_{12}\cr a_{21}&a_{22}
\end{pmatrix}\text{ and }\det(A)\neq 0.
\end{equation}
Hence, by our assumption, the curve \begin{align*} 
\psi(\graph{\fa }) &=\set{\psi(\pair x{\fa (x)}):x\in[0,1] } \cr
&= \set{\pair{a_{11}x+a_{21}\fa (x)+b_1} {a_{12}x+a_{22}\fa (x)+b_2}:x\in [0,1]}
\end{align*}
has a nonempty open  arc that lies on the graph of $\fb $. Thus, 
\begin{equation}
\parbox{9.2cm}
{$\hhb(x):=\fb (a_{11}x+a_{21}\fa (x)+b_1)\,$  and\\ 
$\hha(x):=a_{12}x+a_{22}\fa (x)+b_2\,$ are the same polynomials,}
\label{eqtxtbnhTmBxzW}
\end{equation}
because they agree for infinitely many values of $x$. We have that $a_{21}=0$, since otherwise $\hhb$ and $\hha$ would be of degree 49 and degree at most 7, respectively, and this would contradict \eqref{eqtxtbnhTmBxzW}. Since the  coefficients of $x^4$ and $x^3$ in $\hha$ are zero, the same holds in $\hhb$. 
Hence, \eqref{eqfalphdF}, with $\beta$ instead of $\alpha$, and \eqref{eqtxtbnhTmBxzW} yield that
\begin{align}
-\beta\binom{7}{4}a_{11}^4 b_1^3 &+2\beta\binom{6}{4}a_{11}^4b_1^2-\beta\binom{5}{4}a_{11}^4 b_1 \cr
&= -\beta a_{11}^4 b_1(35b_1^2 -30 b_1+5)=0,\text{ and}\label{eqghetbAa} \\
-\beta\binom{7}{3}a_{11}^3 b_1^4 &+2\beta\binom{6}{3}a_{11}^3b_1^3-\beta\binom{5}{3}a_{11}^3 b_1^2 \cr
&= -\beta a_{11}^3 b_1^2(35b_1^2 -40 b_1+10)=0.\label{eqghetbAb}
\end{align}
Since $0\neq \det(A)=a_{11}a_{22}- a_{12}\cdot 0$, none of $a_{11}$ and $a_{22}$ is zero. Neither is $\beta\in(0,1)$. 
In order to show that $b_1=0$, suppose the contrary. By \eqref{eqghetbAa} and \eqref{eqghetbAb}, $35b_1^2 -30 b_1+5=0$ and $(35b_1^2 -30 b_1+5)-(35b_1^2 -40 b_1+10)=10b_1-5=0$. The last equality gives that $b_1=1/2$, which contradicts the first equality. This proves that $b_1=0$. Hence, the constant term in $\hhb$ is 0. Comparing the constant terms in $\hha$ and $\hhb$, we obtain that $b_2=0$. Now, \eqref{eqfalphdF} turns \eqref{eqtxtbnhTmBxzW} into
\begin{align*}
-\beta a_{11}^7x^7&+2\beta a_{11}^6x^6-\beta a_{11}^5 x^5-a_{11}^2x^2+a_{11}x\cr
&= - a_{22}\alpha x^7+2 a_{22}\alpha  x^6-\alpha  a_{22}x^5- a_{22}x^2+ (a_{12}+a_{22})x.
\end{align*}
Comparing the first two terms, we obtain that $\beta a_{11}^7 = a_{22}\alpha = \beta a_{11}^6$. Since $a_{11}\neq 0\neq\beta$, we conclude that $a_{11}=1$. Since the coefficients of $x^2$ are equal, $a_{22}=1$. Finally, the coefficients of $x$ yield that $a_{12}=a_{11}-a_{22}=0$. By the equalities we have obtained, $\psi$ is the identity map, as required. This proves (F\ref{Ffe}).

Finally, we show that the conjunction of (F\ref{Ffa}), (F\ref{Ffb}), and (F\ref{Ffc}) implies (F\ref{Fff}). By (F\ref{Ffc}), the line through $U$ and $V$, denoted by $\ell(U,V)$, and the line
$\ell(V,W)$ are tangent to the graph of $f_\alpha $; see Figure~\ref{figfunc}. Since $f_\alpha $ is concave on $[0,1]$ by (F\ref{Ffa}) and (F\ref{Ffb}), its graph is below these two (and all other) tangent lines. This yields that $f_\alpha (x)<1/2-|x-1/2|$ for all $x\in (0,1)$. 
Next, let $x_0\in(0,1)$. Since $f'_\alpha (0)=1$, there exists an $x_1\in(0,x_0)$ such that $f_\alpha (x_1)>0$. Similarly, $f'_\alpha (1)=-1$ yields an $x_2\in(x_0,1)$ such that $f_\alpha (x_2)>0$. Since $f_\alpha $ is concave on $[x_1,x_2]$, its graph is above the secant through $\pair{x_1}{f_\alpha (x_1)}$ and $\pair{x_2}{f_\alpha (x_2)}$. Thus, $0<f_\alpha (x_0)$ and  (F\ref{Fff}) holds. This completes the proof  of Lemma~\ref{lemmaFczkmjt}.
\end{proof}

\begin{figure}[ht] 
\centerline
{\includegraphics[scale=1.0]{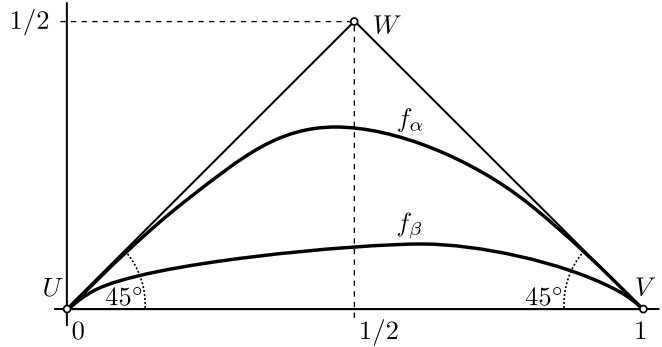}}
\caption{$\fa$ and $\fb$ for $0<\alpha<\beta<1$}  \label{figfunc}
\end{figure}

\section{Proofs and further tools}\label{sectprtools}
\subsection{More about finite convex geometries}

\begin{proof}[Proof of Lemma~\ref{lemmarestrgeom}]
For $X\subseteq E_0$, we have that
\begin{align*}
\Phi_0(\Phi_0(X))&=\Phi(\Phi(X)\cap E_0)\cap E_0\subseteq \Phi(\Phi(X))\cap E_0\cr
&= \Phi(X)\cap E_0 =\Phi_0(X).
\end{align*}
Hence, it is straightforward to see that  $\Phi_0$ satisfies Definition~\ref{defconvgeo}\eqref{defconvgeoa} and \eqref{defconvgeoc}; it suffices to deal only with \ref{defconvgeo}\eqref{defconvgeob}. In order to do so, assume that $A\in \Pow(E_0)$, or $A\in \Pfin(E_0)$,  and  let $x,y\in E_0\setminus \Phi_0(A)$ such that $\Phi_0(A\cup\set x)= \Phi_0(A\cup\set y)$. Since $x\in \Phi_0(A\cup\set x) =  \Phi_0(A\cup\set y)\subseteq  \Phi(A\cup\set y)$,  $A\subseteq  \Phi(A\cup\set y)$, and $\Phi$ is a closure operator, $\Phi(A\cup\set x)\subseteq  \Phi(\Phi(A\cup\set y))= \Phi(A\cup\set y)$. Similarly, $\Phi(A\cup\set y)\subseteq   \Phi(A\cup\set x)$, that is, $\Phi(A\cup\set y) = \Phi(A\cup\set x)$. Clearly, $x,y\notin \Phi(A)$.
Applying ~\ref{defconvgeo}\eqref{defconvgeob} to $\Phi$, it follows that $x=y$, as required.
\end{proof}

\begin{proof}[Proof of Lemma~\ref{lemmaTfCLoCc}] The ``only if'' part is included in Lemma~\ref{lemmarestrgeom}. In order to show the ``if'' part, assume that all  finite restrictions of $\tuple{E;\Phi}$  are convex geometries. By the assumptions of the lemma, 
$\tuple{E;\Phi}$ satisfies \ref{defconvgeo}\eqref{defconvgeoa}. Clearly, it also satisfies \ref{defconvgeo}\eqref{defconvgeoc}.  
If $\tuple{E;\Phi}$ failed to satisfy \ref{defloconvgeo}\eqref{defloconvgeod}  with $X$, $d$, and $d'$, then 
$\restrict{\tuple{E;\Phi}}{X\cup\set{d,d'}}$ would not be a convex geometry.
\end{proof}

Closure operators satisfying \ref{defconvgeo}\eqref{defconvgeoc} will be called \emph{zero-preserving}.   
For a set $E$ and a  subset $\alg G$ of $\Pow(E)$, $\alg G$ is a \emph{zero-preserving closure system} on $E$ if $\emptyset, E\in \alg G$ and $\alg G$ is closed with respect to arbitrary intersections.  As it is well-known, zero-preserving closure systems and zero-preserving closure operators on $E$ mutually determine each other. Namely, the map assigning  $\alg G_\Phi:=\set{X\in \Pow(E): \Phi(X)=X}$ to a zero-preserving closure operator $\Phi$ on $E$ and the map assigning  $\Phi_{\alg G}\colon \Pow(E)\to \Pow(E)$, defined by $\Phi_{\alg G}(X):=\bigcap\set{Y\in \alg G: X\subseteq Y}$, to a zero-preserving  closure system $\alg G$ on $E$ are reciprocal bijections. 

\begin{definition}[Alternative definition of finite convex geometries]\label{defGconvgeo}
We say that $\tuple{E;\alg G}$ is a finite \emph{convex geometry} if $E$ is finite and  $\tuple{E;\Phi_{\alg G}}$ is a convex geometry in the sense of Definition~\ref{defconvgeo}.
\end{definition}

From now on, the paragraph preceding Definition~\ref{defGconvgeo}  enables us to use the notations  $\tuple{ E; \Phi}$ and $\tuple{E;\alg G}$ for the \emph{same} finite convex geometry interchangeably; then $\Phi$ and $\alg G$ are understood as $\Phi_{\alg G}$ and $\alg G_\Phi$, respectively.  The members of $\alg G$ are called the \emph{closed sets} of the convex geometry in question. Note that this abstract concept of closed sets corresponds to the geometric concept of \emph{convex sets}.
As usual, a partial ordering $\leq$ on a set $E$ is \emph{linear} if for every $x,y\in E$, we have $x\leq y$ or $y\leq x$. For simplicity, for a subset $X$ and an element $y$ of $E$,  we will use the notation 
\[ X < y \defiff (\forall x\in X)\,(x<y).
\]

\begin{lemma}[{$(\dagger)$ and Theorems 5.1 and 5.2 in Edelman and Jamison~\cite{edeljam}}]\label{lemmaedJam}\ 

\textup{(A)}  If  $\,\leq_1$, \dots, $\leq_t$ are linear orderings on a finite set $E=\set{1,2,\dots,n}$ and we define $\alg G$ as
\begin{equation}
\alg G:=\set\emptyset\cup \set{X\in \Pow(E): (\forall y\in E\setminus X)\,(\exists i\in\set{1,\dots,t})\,( X <_i y)}, 
\label{eqGhzBfgdF} 
\end{equation}
then $\tuple{E;\alg G}$ is a convex geometry.

\textup{(B)} Every finite convex geometry 
is isomorphic to some $\tuple{E;\alg G}$ such that $\alg G$ is determined  by finitely many linear orderings as in \eqref{eqGhzBfgdF}.
\end{lemma}

Note that essentially the same statement is derived in Adaricheva and Cz\'edli~\cite{kaczg} from a lattice theoretical result. Note also that the
minimum number of linear orderings that we need to represent a finite convex geometry according to (B) is the \emph{convex dimension} of the convex geometry. Finiteness could be dropped from part (A). However, the technical assumption that $E$ is of the form $\set{1,\dots,n}$ will be convenient in the proof of Theorem~\ref{thmmain}.
We will only use part (B). Since its proof is short and we have formulated the above statement a bit differently from \cite{edeljam}, we present the argument below.

\begin{figure}[ht] 
\centerline
{\includegraphics[scale=1.0]{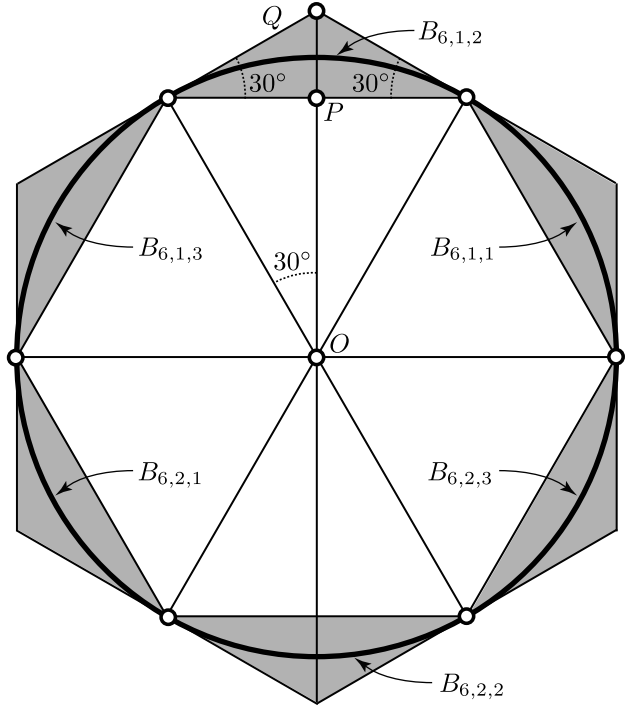}}
\caption{The unit circle with inscribed and circumscribed regular $t$-gons and $mt=6$ circular arcs}
\label{figsmpCr}
\end{figure}

\begin{proof}[Proof of Lemma~\textup{\ref{lemmaedJam}(B)}] 
Let $\tuple{E;\alg G'}$ be a finite convex geometry. Without loss of generality, we can assume that $E=\set{1,\dots,n}$.
With respect to set inclusion, $\alg G'$ is a lattice. Assume that $X\prec Y$ in this lattice, and let $a,b\in Y\setminus X$. Since $Y$ covers $X$, $\Phi(X\cup\set a)=Y=\Phi(X\cup\set b)$.  The anti-exchange property gives that $a=b$, whereby $Y\setminus X$ is a singleton. Hence, for $X,Y\in \alg  G'$,
\begin{equation}
X\prec Y\text{ in the lattice }\alg G' \text{ if{}f }  X\subset Y \text{ and }|Y\setminus X|=1.
\label{eqtxtMxhCvrG}
\end{equation}
Consequently, all maximal chains in $\alg G'$ are of length $n=|E|$. Let $C_1,\dots, C_t$ be a list of all maximal chains in $\alg G'$. By \eqref{eqtxtMxhCvrG}, $C_i$ is of the form
\[
C_i=\bigl\{\emptyset,\set{e_{i,1}}, \set{e_{i,1},e_{i,2}}, \dots,  \set{e_{i,1},e_{i,2},\dots,e_{i,n}} \bigr\},
\]
for $i\in\set{1,\dots,t}$. This allows us to define a linear ordering $\leq_i$ on $E$ as follows:
\[e_{i,1} <_i e_{i,2} <_i e_{i,3} <_i \dots <_i e_{i,n},
\]
for $i\in\set{1,\dots,t}$.
Let $\alg G$ denote what \eqref{eqGhzBfgdF} defines from these linear orderings; we need to show that $\alg G' =\alg G$. Note that $\emptyset\in \alg G\cap \alg G'$. First, assume that $\emptyset\neq X\in \alg G'$. As every element in a finite lattice, $X$ belongs to a maximal chain $C_i$. So $X$ is of the form 
$\set{e_{i,1},e_{i,2},\dots,e_{i,j}}$, and the same subscript $i$ witnesses that $X\in \alg G$. 

Second, assume that $\emptyset\neq X\in \alg G$. For each $y\in E\setminus X$, \eqref{eqGhzBfgdF} allows us to pick an $i=i(y)$ such that $X<_{i(y)} y$. Let $X_y:=\set{e\in E: e<_{i(y)} y}$; it belongs to $C_{i(y)}$, whence $X_y\in\alg G'$. Let $Z:=\bigcap\set{X_y: y\in E\setminus X}$. Since $\alg G'$, like every closure system, is $\bigcap$-closed, $Z\in \alg G'$.   Since all the $X_y$ include $X$, we have that $X\subseteq Z$. For each $y\in E\setminus X$, $y\notin X_y\supseteq Z$ gives that $y\notin Z$. This shows that $X\not\subset Z$. Hence, $X=Z\in \alg G'$, as required.
\end{proof}

\subsection{Almost-circles and their accuracies}

\begin{definition}\label{defFlmCrl}
For integers $t\geq 3$ and $m\in\Nat$ and an $m$-by-$t$ matrix $S=(s_{i,j})_{m\times t}$ of  real numbers from $(0,1)$,  we define a simple closed curve $\acirc {\Fczk} {S}$ as follows; see Figures~\ref{figfunc}--\ref{figcLcL}, where $m=2$, $t=3$, $\alpha<\beta$, and
\begin{equation}
S=\begin{pmatrix}
\alpha&\alpha&\beta\cr
\alpha&\beta&\beta
\end{pmatrix}.
\label{eqdzTnBsmTpxxX}
\end{equation}
Note in advance that $m$ plays the role of some sort of \emph{multiplicity} of the almost-circle we are going to define; it will turn out later that the smaller $\epsilon$ is, the larger multiplicity is needed to achieve the accuracy of $1-\epsilon$.
We start with the unit circle $\set{\pair x y: x^2+y^2=1}$; see Figure~\ref{figsmpCr}.
For $i\in\set{1,\dots,m}$ and $j\in\set{1,\dots,t}$, let $B_{m t,i,j}$ denote the arc of this circle with endpoints
\begin{align}
&\Bpair{\cos \frac{2\pi\cdot(m(i-1)+j-1)}{m t}}{\sin\frac{2\pi\cdot(m(i-1)+j-1)}{m t}}\,\text{ and} \label{eqdlhtpsthRcSa}\\
&\Bpair{\cos \frac{2\pi\cdot(m(i-1)+j)}{m t}}{\sin\frac{2\pi\cdot(m(i-1)+j)}{m t}}.
\label{eqdlhtpsthRcSb}
\end{align}
The secant and the tangent lines of this arc through its endpoints form an isosceles triangle
$B^\triangle_{m t,i,j}$.
These triangles are grey-filled in   Figure~\ref{figsmpCr}. Neither $B^\triangle_{m t,i,j}$, nor $\triangle(U,V,W)$ in  Figure~\ref{figfunc} is a degenerate triangle. Hence, 
for every for $\pair i j\in\set{1,\dots,m}\times \set{1,\dots,t}$, there exists a unique non-degenerate affine transformation $\psi_{m t,i,j}\colon \real^2\to \real^2$ mapping  $\triangle(U,V,W)$ onto $B^\triangle_{m t,i,j}$  such that $V$ and $U$ are mapped to the endpoints 
\eqref{eqdlhtpsthRcSa} and  \eqref{eqdlhtpsthRcSb}, respectively. We let 
\begin{equation}
A_{S,i,j}=\psi_{m t,i,j}(\graph{f_{s_{i,j}}});
\label{eqpsimtiJ}
\end{equation} 
remember that $f_{s_{i,j}}\in \Fczk$ was defined in \eqref{eqfalphdF}; see also Lemma~\ref{lemmaFczkmjt}.
The closed curve formed by these $A_{S,i,j}$ will be denoted by $\acirc {\Fczk} {S}$; see Figure~\ref{figcLcL}.  Note that, in order to increase the visibility of  $\acirc {\Fczk} {S}$, only one of the little triangles is fully grey and two others are partially grey in Figure~\ref{figcLcL}.
\end{definition}

\begin{figure}[ht] 
\centerline
{\includegraphics[scale=1.0]{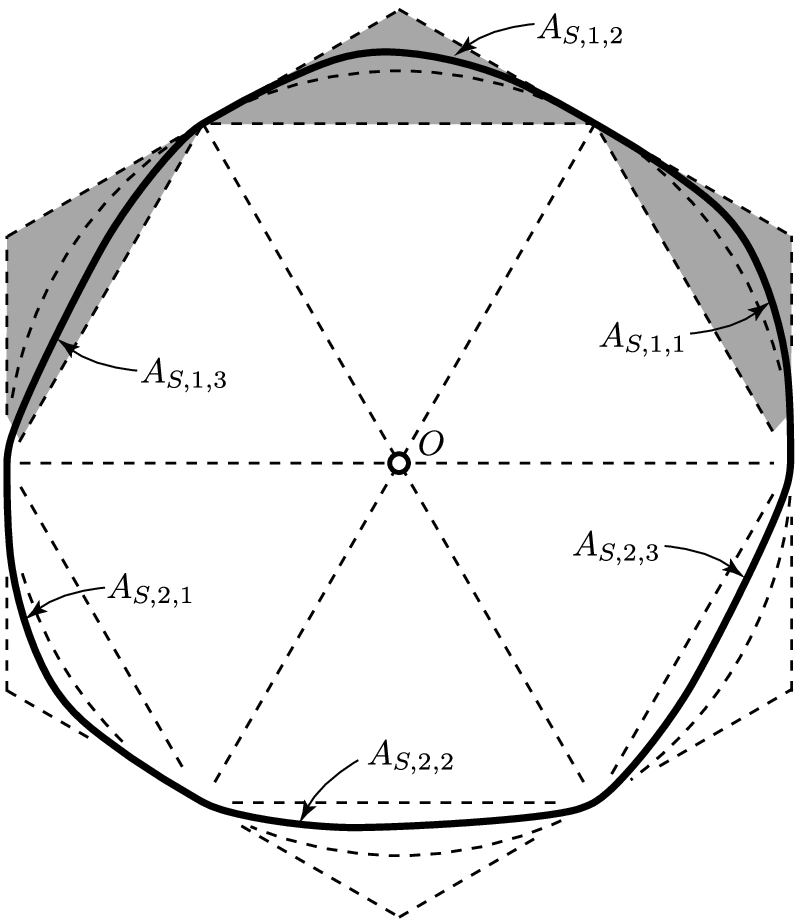}}
\caption{$\acirc {\Fczk} {S}$,  an almost-circle of accuracy $1-(\pi/6)^2\approx  0.7258$}  
\label{figcLcL}
\end{figure}

\begin{lemma}\label{lemmaFlHCr} 
$\acirc {\Fczk}  S$ defined in Definition~\ref{defFlmCrl} is an almost-circle of accuracy $1-(\pi/(mt))^2$.
\end{lemma}

\begin{proof} By (F\ref{Ffc}), the graphs of the $f_{s_{i,j}}$, for $\pair i j\in\set{1,\dots, m}\times \set{1,\dots, t}$, are tangent at their endpoints to the legs of  $\triangle(U,V,W)$; see Figure~\ref{figfunc}.
Since this property is preserved by the transformations $\psi_{m t,i,j}$, see \eqref{eqpsimtiJ}, and a leg of every little grey isosceles triangle in Figures~\ref{figsmpCr} and \ref{figcLcL} lies on the same line as a leg of the next little triangle, it follows that  $\acirc {\Fczk} {S}$ is differentiable where its  arcs, the  $A_{S,i,j}$, are joined. Elementary trigonometry yields that the ratio of $\overline{OP}$ and  $\overline{OQ}$ is $\cos^2(\pi/(mt))=1-\sin^2(\pi/(mt))> 1- (\pi/(mt))^2$; see Figure~\ref{figsmpCr}. 
Therefore, Lemma~\ref{lemmaFczkmjt} and, thus, (F\ref{Ffa})--(F\ref{Fff}) imply Lemma~\ref{lemmaFlHCr} in a straightforward way.
\end{proof}

\subsection{``Concentric'' almost-circles at work}
The title of this subsection only roughly describes its content, because an almost-circle does not have a well-defined center in general. However, only almost circles of the form  $\acirc {\Fczk} {S}$ will occur in this section, and an almost circle $\acirc {\Fczk} {S}$ does have a unique center, which is the center of the corresponding unit circle; see Figure~\ref{figsmpCr}. Actually, we are going to use  concentric almost-circles, which have the same center $\pair 0 0$.
 
We need the following construction; actually, it is a part of the subsequent lemma. This construction shows a lot of similarity with 
that given in Richter and Rogers~\cite{richterrogers}, but we work with curves rather than vertices; actually, the number of our curves will be $m$ times more than the number of their vertices. 
\nothing{Note that we use the original orderings from Edelman and Jamison~\cite{edeljam}; see Lemma~\ref{lemmaedJam} here.}

\begin{definition}\label{defCtRsnL}
For $n\in\Nat$, let 
$\vec o:=\tuple{\leq_1, \dots, \leq_t}$ be a $t$-tuple of linear orderings on the set $E=\set{1,\dots,n}$, and let $\tuple{E;\alg G_{\eqref{eqGhzBfgdF}}}$ be the convex geometry defined in \eqref{eqGhzBfgdF}. Let $m\in\Nat$, the \emph{multiplicity} in the our construction, and let $K$ be a subset of the open interval $(0,1)$ of real numbers such that $|K|=mnt$. 
We order the Cartesian product $\set{1,\dots,m}\times\set{1,\dots,t} \times\set{1,\dots,n}$ lexicographically;  for example,
$\tuple{1,1,3}\lexless3\tuple{1,2,1}$.
The three-fold Cartesian product above and $K$ are of the same size. Hence,  equipped with $\lexless3$, this product is order isomorphic to  $\tuple{K;<}$, where $<$ is the usual ordering of real numbers. Thus, we can write $K$ in the form 
\begin{equation}
\parbox{10cm}{
$K=\{\aijk i j k: \tuple{i,j,k}\in\set{1,\dots,m}\times\set{1,\dots,t}\times\set{1,\dots,n}\}$ such that $\aijk i j k < \aijk {i'}{j'}{k'}$ iff $\tuple{i,j,k}\lexless3 \tuple{i',j',k'}$.}
\label{eqtxtKlxGrrDr}
\end{equation}
For $e\in E$ and $j\in\set{1,\dots,t}$, let 
\begin{equation}
r(j,e)=|\set{x\in E: e\leq_j x}|.
\label{eqpiedF}
\end{equation}
That is, $r(j,e)$ denotes the position of $e$ according to the $j$-th ordering and counted backwards. Associated with $e\in E$, we define an  
\begin{equation}
\parbox{8cm}{$m$-by-$t$ matrix $S(e)=S(m,\vec o,K;e)=(s_{i,j}(e))_{m\times t}$ by the rule
$s_{i,j}(e):= \aijk i j{r(j,e)}\in K$.} 
\label{eqsijkijpie}
\end{equation}
Using the almost-circles $\acirc {\Fczk}  {S(m,\vec o,K;e)}$ constructed in Definition~\ref{defFlmCrl},  we let 
\begin{equation}
\parbox{8.7cm}
{$E_{\eqref{eqtxtdhNBmTrbBcB}}=\set{\acirc {\Fczk}  {S(m,\vec o,K;e)}: e\in E}$  and, with $\Hull{E_{\eqref{eqtxtdhNBmTrbBcB}}}\colon \Pow(E_{\eqref{eqtxtdhNBmTrbBcB}})\to \Pow(E_{\eqref{eqtxtdhNBmTrbBcB}})$ given in \eqref{eqconvTanydef}, we let  $
\alg G_{\eqref{eqtxtdhNBmTrbBcB}}:=\set{X\subseteq \Pow(E_{\eqref{eqtxtdhNBmTrbBcB}}): X=\Hull{E_{\eqref{eqtxtdhNBmTrbBcB}}}(X) }$.}
\label{eqtxtdhNBmTrbBcB}
\end{equation}
\end{definition}

\begin{remark}\label{remuniqueije}
With the notation of Definition~\ref{defCtRsnL}, for each $\beta\in K$, there exists a \emph{unique} triplet 
$\tuple{i,j,e}\in\set{1,\dots,m}\times\set{1,\dots,t}\times E
$
such that $s_{i,j}(e)$, the \hbox{$\pair i j$-th} entry of 
$S(m,\vec o,K;e)$, is $\beta$. Indeed, there is a unique triplet  $\tuple{i,j,k}$ such that $\beta=\aijk i j k$, and, by  \eqref{eqsijkijpie}, $e$ is the $(n+1-k)$-th element of $E$ with respect to $\leq_j$.
\end{remark}

\begin{lemma}\label{lemmagyrPrmTdD}
 With the notation of Definition~\ref{defCtRsnL}, 
$\tuple{E_{\eqref{eqtxtdhNBmTrbBcB}}; \alg G_{\eqref{eqtxtdhNBmTrbBcB}}}$ 
is a convex geometry and it is isomorphic to  $\tuple{E; \alg G_{\eqref{eqGhzBfgdF}}}$.
\end{lemma}

\begin{proof} Letting $H(e)=\acirc {\Fczk}  {S(m,\vec o,K;e)}$, we define a surjective map $H$ from $E=\set{1,\dots,n}$ to $E_{\eqref{eqtxtdhNBmTrbBcB}}$. 
Let $i\in\set{1,\dots,m}$ and $j\in\set{1,\dots,t}$,  and assume that $e_1$ and $e_2$ are  distinct elements of $E$. Then either $e_1<_j e_2$ or  
$e_2<_j e_1$. Hence $r(j,e_1)\neq r(j,e_2)$.
It follows from \eqref{eqtxtKlxGrrDr} and \eqref{eqsijkijpie} that 
$s_{i,j}(e_1)= \aijk  i j {r(j,e_1)} \neq \aijk  i j {r(j,e_2)} = s_{i,j}(e_2)$. Hence, Definition~\ref{defFlmCrl} and (F\ref{Ffe}), or even (F\ref{Ffd}), yield that $H(e_1)$ is distinct from $H(e_2)$; actually, they do not even have an arc in common. Thus, $H$ is injective and it is a bijection.

Next,  we are going to show that, for every $X$,
\begin{equation}
\text{if }X\in \Pow(E)\setminus\alg G_{\eqref{eqGhzBfgdF}}\text{, then }
H(X)\in \Pow(E)\setminus\alg G_{\eqref{eqtxtdhNBmTrbBcB}}.
\label{eqtxtGcmLgpr}
\end{equation}
We know from \eqref{eqGhzBfgdF} that $\emptyset\in\alg G_{\eqref{eqGhzBfgdF}}$, whereby we can assume that $X\neq\emptyset$.
Since $X\notin\alg G_{\eqref{eqGhzBfgdF}}$, \eqref{eqGhzBfgdF} yields a $y\in E\setminus X$ such that for all $j\in\set{1,\dots,t}$, $X\nless_j y$. 
Hence, for each $j\in\set{1,\dots,t}$, we can pick an $e_j\in X$ such that $y<_j e_j$. 
By \eqref{eqpiedF}, $r(j,y)>r(j,e_j)$. Hence, 
\eqref{eqtxtKlxGrrDr}   and \eqref{eqsijkijpie} give that $s_{i,j}(y)>s_{i,j}(e_j)$ holds for all $i\in\set{1,\dots, m}$ and $j\in\set{1,\dots,t}$. By (F\ref{Ffd}) and  Definition~\ref{defFlmCrl}, the $\pair i j$-th arc
 $A_{S(m,\vec o,K;y),i,j}$ of 
 $H(y)=\acirc {\Fczk}  {S(m,\vec o,K;y)}$  is closer to the center $O:=\pair 0 0$ of the unit circle than the $\pair i j$-th arc
 $A_{S(m,\vec o,K;e_j),i,j}$ of 
 $H(e_j)=\acirc {\Fczk}  {S(m,\vec o,K;e_j)}$.
Here, ``closer'' means that only the endpoints are in common but for every inner point $P$ of the second arc, the line segment $\overline{OP}$ has an inner point lying on the interior of the first arc. Therefore, 
$H(y)\in \Hull{E_{\eqref{eqtxtdhNBmTrbBcB}}}\set{H(e_j): j\in\set{1,\dots,t}} \subseteq \Hull{E_{\eqref{eqtxtdhNBmTrbBcB}}}(H(X))$. However, $y\notin X$  and the injectivity of $H$ give that $H(y)\notin H(X)$. This indicates that $H(X)$ is not closed with respect to $\Hull{E_{\eqref{eqtxtdhNBmTrbBcB}}}$. Thus, $H(X)\notin \alg G_{\eqref{eqtxtdhNBmTrbBcB}}$, proving \eqref{eqtxtGcmLgpr}.

Next, we are going to show converse implication, that is, 
\begin{equation}
\text{if }X\in \alg G_{\eqref{eqGhzBfgdF}}\text{, then }
H(X)\in \alg G_{\eqref{eqtxtdhNBmTrbBcB}}.
\label{eqtxtGcHjTqgr}
\end{equation}
Assume that $X\in\alg G_{\eqref{eqGhzBfgdF}}$. In order to verify \eqref{eqtxtGcHjTqgr}, we need to show that
for every $y'\in E_{\eqref{eqtxtdhNBmTrbBcB}}\setminus H(X)$, we have that  $y'\notin \Hull{E_{\eqref{eqtxtdhNBmTrbBcB}}}(H(X))$. Since $H$ is a bijection, $y'=H(y)$ for a uniquely determined $y\in E\setminus X$. Applying \eqref{eqGhzBfgdF} to this $y$, we obtain  a $j\in\set{1,\dots,t}$ such that $X<_j y$. Hence, for every $e\in X$, $e<_j y$. Thus, \eqref{eqpiedF} gives that $r(j,e)>r(j,y)$. Combining this with  \eqref{eqtxtKlxGrrDr}  and \eqref{eqsijkijpie}, we conclude that, for all $i\in\set{1,\dots, m}$, $s_{i,j}(e) > s_{i,j}(y)$.
(Actually, one such $i$ is sufficient in the present argument.)  
By (F\ref{Ffd}) and  Definition~\ref{defFlmCrl}, the $\pair i j$-th arc  $A_{S(m,\vec o,K;e),i,j}$ of 
 $H(e)=\acirc {\Fczk}  {S(m,\vec o,K;e)}$  is closer to the center $O:=\pair 0 0$ than the $\pair i j$-th arc
 $A_{S(m,\vec o,K;y),i,j}$ of 
 $H(y)=\acirc {\Fczk}  {S(m,\vec o,K;y)}$. 
This means that the $\pair i j$-th arc of $H(y)$ is ``outside''  $H(e)$ for all $e\in X$.  That is, the  $\pair i j$-th curve of $H(y)$ is outside all members of $H(X)$.    Consequently, $y'=H(y)  \notin \Hull{E_{\eqref{eqtxtdhNBmTrbBcB}}}(H(X))$, as required.
This proves \eqref{eqtxtGcHjTqgr}. 
Finally, \eqref{eqtxtGcmLgpr} and \eqref{eqtxtGcHjTqgr} imply that the bijection $H$ is actually an isomorphism, proving Lemma~\ref{lemmagyrPrmTdD}.
\end{proof}

\subsection{The rest of the proof}
We are going to define an appropriate set $\Tczk$ needed by Theorem~\ref{thmmain}.
Consider the set $U$ of all triplets $\tuple{n,\vec o,m}$ where 
\begin{enumeratei}
\item  $n$ and $m$ are positive integers;
\item $\vec o=\tuple{\leq_1,\leq_2,\dots,\leq_t}$ is a nonempty tuple of finitely many linear orderings on the set $E:=\set{1,\dots,n}$; the number of its components is $\dim(\vec o)=t\in\Nat$.
\end{enumeratei}
Since $U$ is a countably infinite set,
we obtain from basic cardinal arithmetic that $|\real|\cdot|U|=|(0,1)|$. This allows us to partition the real interval $(0,1)$ as a  union 
$(0,1)=\bigcup\set{I_{n,\vec o,m}: \tuple{n,\vec o,m}\in U}
$
of pairwise disjoint subsets such that  $|I_{n,\vec o,m}|=2^{\aleph_0}$ for all $\tuple{n,\vec o,m}\in U$. Note that 
if we wrote the elements of $(0,1)$ into unique decimal forms not ending with $999\dots$ (infinitely many), then Cantor's well-known method together with lots of technicalities would allow us to define $I_{n,\vec o,m}$ and, eventually, $\Tczk$ uniquely. However, we do not seek  uniqueness; our only goal is to prove the \emph{existence} of an appropriate $\Tczk$. In the next step, using the equality $|\real|\cdot n\cdot\dim(\vec o)\cdot m=2^{\aleph_0}=|I_{n,\vec o,m}|$, which is trivial from cardinal arithmetic, we can partition $I_{n,\vec o,m}$ as the union
$I_{n,\vec o,m}=\bigcup\set{K_{\kappa,n,\vec o,m}: \tuple{\kappa,\tuple{n,\vec o, m}}   \in\real\times U}
$
of pairwise disjoint $n\cdot\dim(\vec o)\cdot m$-element subsets of $I_{n,\vec o,m}$. In order to ease the notation, we will write $\tuple{\kappa,n,\vec o, m}$ rather than $\tuple{\kappa,\tuple{n,\vec o, m}}$. Then, clearly, 
\begin{equation}
\parbox{9.3cm}{for every $\tuple{\kappa,n,\vec o,m}\in\real\times U$, $K_{\kappa,n,\vec o,m}\subset(0,1)$ and  $|K_{\kappa,n,\vec o,m}|=n\cdot\dim(\vec o)\cdot m$, and whenever $\tuple{\kappa,n,\vec o,m}\neq \tuple{\kappa',n',\vvec o',m'}$, then $K_{\kappa,n,\vec o,m}$ is disjoint from $K_{\kappa',n',\vvec o',m'}$.}
\label{eqtxtKlyNlTt}
\end{equation}
Next, with the almost-circles $\acirc {\Fczk}  {S(m,\vec o,K_{\kappa,n,\vec o,m};e)}$ from \eqref{eqtxtdhNBmTrbBcB}, we let
\begin{equation}
\begin{aligned}
\Tczk :=&\{\psi\bigl( \acirc {\Fczk}  {S(m,\vec o,K_{\kappa,n,\vec o,m};e)} \bigr): \tuple{\kappa,n,\vec o,m}\in\real\times U,\cr
&\phantom{\{}e\in\set{1,\dots,n}, \text{ and } \psi\colon\real^2\to\real^2\text{ is a non-degenerate }\cr
&\phantom{\{}\text{affine transformation}\}\,
\cup \,\set{\set{\vec p}:\vec p\in\real^2} . 
\end{aligned}
\label{eqzGhBmTxlXq}
\end{equation}
Now, we are in the position to prove our theorem.

\begin{proof}[Proof of Theorem~\ref{thmmain}]
The $\acirc {\Fczk}  {S(m,\vec o,K_{\kappa,n,\vec o,m};e)}$ in \eqref{eqzGhBmTxlXq} are almost circles by Lemma~\ref{lemmaFlHCr}. Hence, by Definition~\ref{defAmstCrlE}, they are 
differentiable convex simple closed planar curves. So are their images by non-degenerate affine transformations, proving  part \eqref{thmmaina} of the theorem.

Part \eqref{thmmainc} is a trivial consequence of \eqref{eqzGhBmTxlXq}, since the composite of two non-degenerate affine transformations is again a non-degenerate affine transformation.

In order to prove part \eqref{thmmaind}, take an arbitrary finite convex geometry and a positive $\epsilon<1$. By Lemma~\ref{lemmaedJam}, 
we can assume that this convex geometry is given on a set $\set{1,\dots, n}$ with the help of a $\dim(\vec o)$-tuple $\vec o$ of linear orderings on $\set{1,\dots, n}$; see \eqref{eqGhzBfgdF}. Pick an $m\in\Nat$ such that $m \geq\pi\cdot \dim(\vvec o)^{-1}\epsilon^{-1/2}$, and let $\kappa\in\real$. We also need $K_{\kappa,n,\vec o,m}$; see  \eqref{eqtxtKlyNlTt}. 
We apply Definition~\ref{defCtRsnL} and, in particular, \eqref{eqtxtdhNBmTrbBcB}, to 
$\tuple{n,m,\vec o,K_{\kappa,n,\vec o,m}}$
instead of $\tuple{n,m,\vec o,K}$ in order to obtain $\tuple{E^{(\kappa,n,\vec o,m )}_{\eqref{eqtxtdhNBmTrbBcB}},\alg G^{(\kappa,n,\vec o,m )}_{\eqref{eqtxtdhNBmTrbBcB}}}$. Here, 
\begin{equation}
E^{(\kappa,n,\vec o,m )}_{\eqref{eqtxtdhNBmTrbBcB}} = \set{\acirc {\Fczk}  {S(m,\vec o,K_{\kappa,n,\vec o,m};e)}: e\in\set{1,\dots,n}}.
\label{eqduGnbTrWsG}
\end{equation}
By Lemma~\ref{lemmaFlHCr}, these almost-circles are of accuracy 
\begin{align*}
1-(\pi/(m\cdot \dim(\vvec o))^2 &\geq 1- (\pi/(\pi\cdot \dim(\vvec o)^{-1}\epsilon^{-1/2}\cdot\dim(\vvec o ))^2\cr
&=1-\epsilon.
\end{align*}
By Lemma~\ref{lemmagyrPrmTdD},  
$\tuple{E^{(\kappa,n,\vec o,m )}_{\eqref{eqtxtdhNBmTrbBcB}},\alg G^{(\kappa,n,\vec o,m )}_{\eqref{eqtxtdhNBmTrbBcB}}}$ is isomorphic to  the arbitrary convex geometry we started with. This proves the first half of part \eqref{thmmaind} of the Theorem. 
In order to prove the second half, suppose for a contradiction that $\tuple{\kappa,n,\vec o,m}$ and $\tuple{\kappa',n',\vvec o',m'}$ are distinct quadruples of $\real\times U$ but
$E^{(\kappa,n,\vec o,m )}_{\eqref{eqtxtdhNBmTrbBcB}}$ is not affine-disjoint from $E^{(\kappa',n',\vvec o',m')}_{\eqref{eqtxtdhNBmTrbBcB}}$. Hence, there is an almost-circle $C_1$ in the first set and a non-degenerate affine transformation $\psi\colon \real^2\to\real^2$ such that $C_2:=\psi(C_1)$ belongs to the second set. 
Since $C_1$ is one of the almost-circles occurring in \eqref{eqduGnbTrWsG}, we know from \eqref{eqpsimtiJ} that its  arcs are of  the form $\psi_{m\dim(\vvec o),i,j}(\graph{f_{s_{i,j}}})$. 
By \eqref{eqsijkijpie} and  \eqref{eqduGnbTrWsG}, $\alpha:=s_{i,j}$ belongs to 
$K_{\kappa,n,\vec o,m}$. Hence, the arcs of $C_1$ are non-degenerate affine images of finitely many
graphs $\graph{f_{\alpha_1}}$, $\graph{f_{\alpha_2}}$, \dots{}   such that $\alpha_1,\alpha_2,\dots$ belong to $K_{\kappa,n,\vec o,m}$. 
By the same reason, 
the arcs of $C_2$ are non-degenerate affine images of finitely many
graphs $\graph{f_{\beta_1}}$, $\graph{f_{\beta_2}}$, \dots with $\beta_1, \beta_2,\dots$ belonging to $K_{\kappa',n',\vvec o',m'}$. Since $\tuple{\kappa,n,\vec o,m}$ and $\tuple{\kappa',n',\vvec o',m'}$ are distinct quadruples, we know from \eqref{eqtxtKlyNlTt} that $\set{\alpha_1,\alpha_2,\dots}$ is disjoint from $\set{\beta_1,\beta_2,\dots}$. Since $\psi(C_1)=C_2$, $\psi(\graph{f_{\alpha_1}})$ and some of $\psi(\graph{f_{\beta_1}})$, $\psi(\graph{f_{\beta_2}})$,\dots have a nonempty open arc in common. Since $\alpha_1\notin\set{\beta_1,\beta_2,\dots}$, this common arc contradicts (F\ref{Ffe}). 
Therefore, part \eqref{thmmaind} of the Theorem holds. 

Next, before dealing with  part \eqref{thmmainb}, we show that, for every $Y\subseteq \Tczk$,
\begin{equation}
\rhullo2(\Pnt(Y)) =
\rhullo2(\Pnt(\Hull\Tczk(Y))).
\label{eqcNvPnthnbnD}
\end{equation}
Using  \eqref{eqconvTanydef}, we have that 
\begin{align*}
\Pnt&(\Hull\Tczk(Y))\overset{\eqref{eqconvTanydef}}{=} \Pnt(\set{D\in\Tczk: D\subseteq \rhullo2(\Pnt(Y))} )\cr
&\subseteq \Pnt(\set{\rhullo2(\Pnt(Y))}) = \rhullo2(\Pnt(Y)).
\end{align*}
Applying $\rhullo2$ to both sides and using that $\rhullo2\circ\rhullo2 =\rhullo2$, 
\[\rhullo2(\Pnt(\Hull\Tczk(Y))) \subseteq \rhullo2(\Pnt(Y)).
\]
The converse inclusion also holds, because $\Hull\Tczk(Y)\supseteq Y$. This proves \eqref{eqcNvPnthnbnD}.

\begin{figure}[ht] 
\centerline
{\includegraphics[scale=1.0]{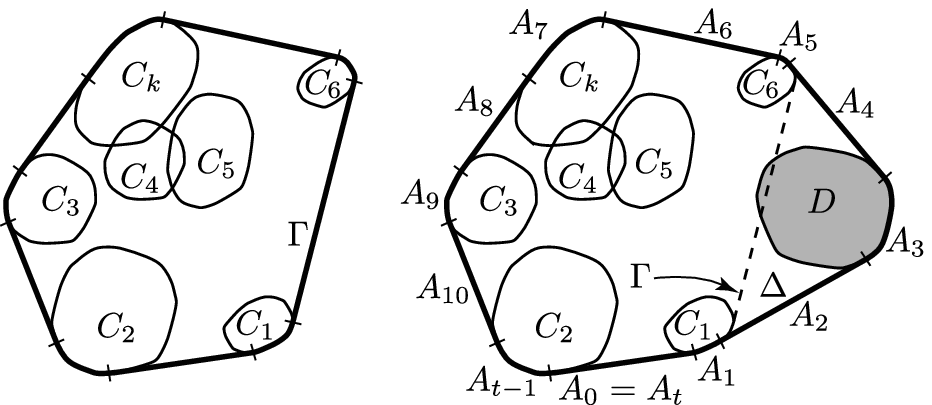}}
\caption{$\Hull\Tczk(\set{C_1,\dots,C_k,D})$ determines $D$, provided $D\notin\Hull\Tczk(\set{C_1,\dots,C_k})$}
\label{figtmstCrZLLs}
\end{figure}

\nothing{From several aspects, the proof of part \eqref{thmmainb} begins in the same way as that of  Cz\'edli~\cite[Proposition 2.1]{czgcircles} for circles. Hence, where the argument is the same as in \cite{czgcircles}, we give less details here.} 

From some aspects,  the proof of part \eqref{thmmainb} is analogous to that of   Cz\'edli~\cite[Proposition 2.1]{czgcircles} for circles.
If $D\in X\subseteq \Tczk$, then $D\subseteq\Pnt(X)\subseteq\rhullo2(\Pnt(X))$ gives that $D\in\Hull\Tczk(X)$. Hence, $X\subseteq \Hull\Tczk(X)$. Obviously, $\Hull\Tczk$ is monotone and zero-preserving. If $D\in\Hull\Tczk(\Hull\Tczk(X))$, then $D\in \Tczk$ and 
\begin{align*}
D\overset{\eqref{eqconvTanydef}}\subseteq 
\rhullo2(\Pnt(\Hull\Tczk(X))) \overset{\eqref{eqcNvPnthnbnD}}= \rhullo2(\Pnt(X)),
\end{align*}
whereby $D\in \Hull\Tczk(X)$. Hence, $\Hull\Tczk(\Hull\Tczk(X))\subseteq \Hull\Tczk(X)$, and $\Hull\Tczk$ is a zero-preserving closure operator. 
That is,  $\tuple{\Tczk;\Hull\Tczk}$ satisfies Definition~\ref{defconvgeo}\eqref{defconvgeoa} and \eqref{defconvgeoc}. In order to show that it also satisfies Definition~\ref{defloconvgeo}\eqref{defloconvgeod}, let $X=\set{C_1,\dots,C_t}\subseteq \Tczk$, $D,D'\in\Tczk$, and assume that $\Hull\Tczk(X\cup\set D)=\Hull\Tczk(X\cup\set {D'})$ and  none of $D$ and $D'$ belongs to $\Hull\Tczk(X)$. We need to give an affirmative answer to the question 
\begin{equation} 
\text{does $\,D = D'\,$ hold?}
\label{eqtxtnEEdthsWh}
\end{equation}
\nothing{Note that \cite{czgcircles} allows circles of radius 0, that is, singleton sets. If $D$ or $D'$ is a singleton, then argument given in  \cite{czgcircles} works without any essential change. Thus, we can assume that neither $D$, nor $D'$ is a singleton.} 
By \eqref{eqconvTanydef}, none of $D$ and $D'$ is a subset of $\rhullo2(\Pnt (X))$. 
Let $\Gamma$ be the boundary of $\rhullo2(\Pnt (X))$; see Figures~\ref{figtmstCrZLLs} and \ref{figtmstDntWchrcs}. Note that in these 
two figures, $D$ is grey-filled but no $D'$ distinct from $D$ is depicted. Singleton members of $X$ like $C_2$ in Figure~ \ref{figtmstDntWchrcs} cause no problem.

\begin{figure}[ht] 
\centerline
{\includegraphics[scale=1.0]{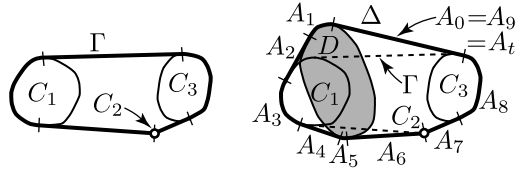}}
\caption{Two characteristic arcs, $A_1$ and $A_{5}$}
\label{figtmstDntWchrcs}
\end{figure}

We think of $\Gamma$ as a tight resilient rubber noose. Since  $D\nsubseteq\rhullo2(\Pnt (X))$, $D$ pushes $\Gamma$ outwards to obtain the boundary $\Delta$ of $\rhullo2{(\Pnt(X\cup\set D))}$.  It follows from \eqref{eqcNvPnthnbnD} that $\Delta$ is also the boundary of $\rhullo2{(\Pnt(X\cup\set {D'}))}$.
Hence, $D'$ also pushes $\Gamma$ outwards to $\Delta$. Observe that 
$\Delta$ can be decomposed into arcs  $A_0,A_1, \dots, A_{t-1}, A_t=A_0$ of positive lengths; see Figures~\ref{figtmstCrZLLs} and \ref{figtmstDntWchrcs}, and the same holds (with different $t$) for $\Gamma$.  Keep in mind that our terminology concerning arcs allows straight line segments as special arcs. When distinction is necessary, we speak of \emph{straight line segments} and \emph{non-straight arcs}.

First, assume that one of $D$ and $D'$ is a singleton. Let, say, $D$ be a singleton. Instead of a separate figure, take  $X=\set{C_1,C_3}$ and $D=C_2$ on the left of Figure~\ref{figtmstDntWchrcs} to see an example. Clearly, two straight line segments of $\Delta$, none of them being a part of a $\Gamma$-arc, form an angle with vertex $D$. This fact makes $D$ recognizable from $\Gamma$ and $\Delta$, and it follows that $D=D'$. Hence, in the rest of the proof, we can assume that none of $D$ and $D'$ is a singleton.

A non-straight arc of $\Delta$ will be called a \emph{characteristic arc} (with respect to $\Gamma$) if it is neither a  straight line segment, nor a subset of an arc of $\Gamma$. Necessarily, a characteristic arc is also an arc of $D$, and the same holds for $D'$. Since $\Delta\neq \Gamma$ and none of $D$ and $D'$ is a singleton, there is at least one characteristic arc. For example, the only characteristic arc in Figure~\ref{figtmstCrZLLs} is $A_3$, but there are two characteristic arcs, $A_1$ and $A_5$, in Figure~\ref{figtmstDntWchrcs}. Thus, $D$ and $D'$ have an arc of positive length in common: a characteristic arc $I_0$  of $\Delta$. By \eqref{eqzGhBmTxlXq}, there exist $\tuple{\kappa,n,\vec o,m}$ and 
$\tuple{\kappa',n',\vvec o',m'}$ in $\real\times U$, $e\in\set{1,\dots,n}$, $e'\in\set{1,\dots,n'}$, and  non-degenerate affine transformations $\mu$ and $\mu'$ such that
\begin{equation}
\begin{aligned}
D&=\mu\bigl(\acirc {\Fczk}  {S(m,\vec o,K_{\kappa,n,\vec o,m};e)}\bigr)\quad\text{and }\cr
D'&=\mu'\bigl(\acirc {\Fczk}  {S(m',\vvec o',K_{\kappa',n',\vvec o',m'};e')}\bigr).
\end{aligned}
\label{eqDnndDdDpr}
\end{equation}
Since  $D$ and $D'$ have the arc $I_0$ in common, their preimages, 
\[
\text{$C:=\acirc {\Fczk}  {S(m,\vec o,K_{\kappa,n,\vec o,m};e)}$ and $C':=\acirc {\Fczk}  {S(m',\vvec o',K_{\kappa',n',\vvec o',m'};e')}$,}
\]
have arcs 
\begin{equation}
A:=A_{S(m,\vec o,K_{\kappa,n,\vec o,m};e),i,j}\,\text{ and }\, 
A':=A_{S(m',\vvec o',K_{\kappa',n',\vvec o',m'};e'),i',j'},
\label{eqAdndundApr}
\end{equation}
respectively and in the sense of \eqref{eqpsimtiJ},  such that the $I_1\subseteq \mu(A)$ and $I_1\subseteq\mu'(A')$ holds for some nonempty open sub-arc $I_1$ of $I_0$; necessarily of positive length. By the construction of our almost-circles in Definition~\ref{defFlmCrl}, there are $\alpha,\alpha'\in (0,1)$ and non-degenerate affine transformations $\phi$ and $\phi'$ such that $A=\phi(\graph{f_{\alpha}})$ and $A'=\phi'(\graph{f_{\alpha'}})$. Letting $\psi:=\mu\circ\phi$ and $\psi':=\mu'\circ\phi'$, we have that $I_1$ is a common open arc of  both $\psi(\graph{f_{\alpha}})$ and $\psi'(\graph{f_{\alpha'}})$. 
It follows from (F\ref{Ffe}) that $\alpha=\alpha'$ and $\psi=\psi'$. 

Roughly saying, disregarding $\psi$ in \eqref{eqzGhBmTxlXq}, $f_\alpha \in \Tczk$ is used only once in the definition of $\Tczk$, that is only in one almost-circle and only at one edge of this almost-circle; this implies that $\phi=\phi'$.
However, we give rigorous details below.

By the construction, see Definition~\ref{defFlmCrl}, \eqref{eqsijkijpie}, and \eqref{eqzGhBmTxlXq}, $\alpha \in K_{\kappa,n,\vec o,m}$ is an entry of the matrix $S:=S(m,\vec o,K_{\kappa,n,\vec o,m};e)$  and $\alpha'\in  K_{\kappa',n',\vvec o',m'}$ is that of $S':=S(m',\vvec o,'K_{\kappa',n',\vvec o',m'};e')$. Since $\alpha=\alpha'$, \eqref{eqtxtKlyNlTt} yields that the quadruples $\tuple{\kappa,n,\vec o,m}$ and $\tuple{\kappa',n',\vvec o',m'}$ are the same. Hence, using the equality $\alpha=\alpha'$ together with  Remark~\ref{remuniqueije} 
for $\tuple{n,m,\vec o,K_{\kappa,n,\vec o,m}}$
in the role of  $\tuple{n,m,\vec o,K}$, we conclude that $S=S'$, and there is a unique triplet  $\tuple{i,j,e}$ such that $\alpha$ is the    $\pair i j$-th entry of $S$. Furthermore, $S=S'$ and $e=e'$ give that $C=C'$. Taking the uniqueness of $\pair i j$ also into account, we obtain that $A=A'$.  The position of the \hbox{$\pair i j$-th} arc in $C=C'$ is uniquely determined by its endpoints given in \eqref{eqdlhtpsthRcSa} and  \eqref{eqdlhtpsthRcSb}. Therefore, $\phi=\phi'$. 

Finally, multiplying $\mu\circ \phi =\psi=\psi'=\mu'\circ \phi'=\mu'\circ \phi$ by $\phi^{-1}$ from the right, we obtain that $\mu=\mu'$.
Armed with $\mu=\mu'$ and $C=C'$, \eqref{eqDnndDdDpr} gives that $D=D'$, as required in \eqref{eqtxtnEEdthsWh}. Thus, $\Tczk$  satisfies Definition~\ref{defloconvgeo}\eqref{defloconvgeod}. Hence, part \eqref{thmmainb} of the theorem holds and the proof if complete.
\end{proof}

\end{document}